\documentclass{compositio}

\usepackage[latin1]{inputenc}
\usepackage[T1]{fontenc}
\usepackage{amsmath,amsfonts,amssymb,amsthm}
\usepackage{graphicx}
\usepackage[all]{xy}

\theoremstyle{plain}
\newtheorem{theorem}{Theorem}[section]
\newtheorem{lemma}[theorem]{Lemma}
\newtheorem{prop}[theorem]{Proposition}
\newtheorem{cor}[theorem]{Corollary}

\theoremstyle{definition}
\newtheorem{defin}[theorem]{Definition}

\theoremstyle{remark}
\newtheorem{rmk}[theorem]{Remark}

\newcommand{\ol}[1]{\overline{#1}}

\renewcommand{\hat}[1]{\widehat{#1}}
\renewcommand{\tilde}[1]{\widetilde{#1}}

\newcommand{\set}[1]{{\left\{#1\right\}}}
\newcommand{\pa}[1]{{\left(#1\right)}}

\newcommand{\abs}[1]{{\left|#1\right|}}

\newcommand{\ra}{\rightarrow}

\newcommand{\sheaf}[2]{\mathcal{#1}_{#2}}
\newcommand{\fascio}[1]{\sheaf{O}{#1}}
\newcommand{\isom}{\xrightarrow{\sim}}
\newcommand{\prym}{\mathcal{P}}

\DeclareMathOperator{\rk}{rk}
\DeclareMathOperator{\Hom}{Hom}

\DeclareMathOperator{\sym}{Sym}
\DeclareMathOperator{\ns}{NS}
\DeclareMathOperator{\codim}{codim}
\DeclareMathOperator{\supp}{supp}
\DeclareMathOperator{\spec}{Spec}

\DeclareMathOperator{\pic}{Pic}
\DeclareMathOperator{\ext}{Ext}
\DeclareMathOperator{\cliff}{Cliff}

\title[Generic Torelli theorem for Prym varieties]{Generic Torelli theorem for Prym varieties of ramified coverings}

\author{Valeria Ornella Marcucci}
\email{valeria.marcucci@unipv.it}
\address{Dipartimento di Matematica ``F. Casorati''\\
 	Universit\`a di Pavia\\
	via Ferrata 1, 27100 Pavia, Italy}

\author{Gian Pietro Pirola}
\email{gianpietro.pirola@unipv.it}
\address{Dipartimento di Matematica ``F. Casorati''\\
 	Universit\`a di Pavia\\
	via Ferrata 1, 27100 Pavia, Italy}


\thanks{This work has been partially supported by 1) FAR 2010 (PV) \emph{"Variet\`a algebriche, calcolo algebrico, grafi orientati e topologici"}; 2) INdAM (GNSAGA) 3) PRIN 2009 \emph{``Moduli, strutture geometriche e loro applicazioni''}}

\date{\today}

\begin{document}

\begin{abstract}
We consider the Prym map from the space of double coverings of a curve of genus $g$ with $r$ branch points to the moduli space of abelian varieties. We prove that $\prym\colon \mathcal{R}_{g,r} \ra \mathcal{A}_{g-1+\frac{r}{2}}^\delta$ is generically injective if
\[
r>6\text{ and }g\geq 2,\quad r=6\text{ and }g\geq 3,\quad r=4\text{ and }g\geq 5,\quad r=2\text{ and }g\geq 6.
\]
We also show that a very general Prym variety of dimension at least $4$ is not isogenous to a Jacobian.
\end{abstract}

\maketitle

\section{Introduction}
\label{intr}

Let $C$ be a complex projective curve of genus $g\geq 1$ and $\pi\colon D\ra C$ be a $2$-sheeted covering of $C$ ramified at $r$ points. The \emph{Prym variety} $P$ of $\pi$ is the identity component of the kernel of the norm homomorphism
\[
N\colon J\pa{D} \ra J\pa{C}.
\]
Note that if $r\neq 0$, then $\ker N$ is connected (see for example \cite[Section 3, Lemma]{prymMumford} and \cite[Section 1, Lemma 1.1]{HurwitzSpacesKanev}). The Prym variety $P$ of $\pi$ is an abelian variety of dimension $g-1+\frac{r}{2}$ and the divisor $\Theta_P:= \Theta_{J\pa{D}}\cap P$ gives a polarization of type
\begin{equation}
\label{eq:polarizzazione}
\delta=(1, \ldots, 1,\underset{g}{\underbrace{2,\ldots, 2}})
\end{equation}
(see, for instance, \cite[Section 1, Lemma 1.1]{HurwitzSpacesKanev}).

In this paper we deal with the generic Torelli theorem for Prym varieties of ramified coverings. The infinitesimal Torelli theorem stated in \cite{pol1122} (see also Proposition \ref{prop:proplocaltorelli}) and the \'etale case suggest that the result should hold when the dimension of the space of coverings is strictly smaller than the dimension of the moduli space of abelian varieties.

Let $C$ be a curve, $\eta$ a line bundle on $C$ and $R$ a multiplicity free divisor in the linear system $\abs{\eta^2}$. Following \cite{prymMumford} we can provide the coherent $\fascio{C}$-module $\fascio{C}\oplus \eta^{-1}$ with a natural $\fascio{C}$-algebra structure depending on $R$. The natural projection
\[
\pi\colon D:=\spec\pa{\fascio{C}\oplus \eta^{-1}}\ra C=\spec\pa{\fascio{C}}
\]
is a ramified double covering with branch points in the support of $R$ and, vice versa, each double covering $\pi$ comes in this form.
Let $\mathcal{R}_{g,r}$ denote the scheme parametrizing triples $\pa{C, \eta, R}$ up to isomorphism; the \emph{Prym map} is the morphism
\[
\prym\colon \mathcal{R}_{g,r} \ra \mathcal{A}_{g-1+\frac{r}{2}}^\delta
\]
which associates to $\pa{C, \eta, R}$ the Prym variety $P$ of $\pi$. We call the \emph{Prym locus} the closure of the set of Prym varieties in $\mathcal{A}_{g-1+\frac{r}{2}}^\delta$ and we denote it by $\mathcal{P}_{g-1+\frac{r}{2}}^\delta$.

In the \'etale case the Prym map is generically finite for $g\geq 6$ (Wirtinger theorem, see \cite{Beauv77}) and it is never injective (\cite{DS81, donagiTetragonal, narbiellipt, narPositive}). Kanev (\cite{KanevTorelli}) and Friedman and Smith (\cite{FriedmanRoyPrym}) proved independently that the Prym map has generically degree one for $g\geq 7$. This is the so-called \emph{Generic Torelli theorem for Prym varieties}. It is known that (see \cite{pol1122, LangeSernesi, GreenLaz} and Section \ref{subsec:finitness}) the Prym map is generically finite (onto its image) if and only if
\[
\dim \mathcal{R}_{g,r} \leq \dim \mathcal{A}_{g-1+\frac{r}{2}}^\delta.
\]
In \cite{pol1122}, Nagaraj and Ramanan proved that, if
\[
r=4, \qquad g\geq 4, \qquad h^0(\eta^2)=1,
\]
the triple $\pa{C, \eta, R}$ can be recovered from the Prym variety $\prym\pa{C, \eta, R}$. Furthermore the Prym map
\[
\prym\colon \mathcal{R}_{3,4} \ra \mathcal{A}_{4}^\delta
\]
is a dominant morphism of degree $3$ (see \cite[Theorem 9.14]{pol1122} and \cite[Theorem 5.11]{bardcilverra}).
We prove that:

\begin{theorem}[Generic Torelli]
\label{theorem: globaltorelli}
If one of the following conditions holds
\[
r>6\text{ and }g\geq 2,\qquad r=6\text{ and }g\geq 3,\qquad r=4\text{ and }g\geq 4,\qquad r=2\text{ and }g\geq 6,
\]
the Prym map is generically injective.
\end{theorem}

This shows that our expectation is essentially satisfied. In the bi-elliptic case ($r\geq 6$ and $g=1$) the generic injectivity of the Prym map has been recently proved by the first author and J.C. Naranjo (see \cite{MarNar}) by using the techniques we develop in this paper and a construction of Del Centina and Recillas (see \cite{DelCentinaRecillas}). Thus there are only two cases left ($r=2$ and $g=5$, $r=6$ and $g=2$) that will be object of future studies.

The proof consists of two steps. In the spirit of \cite{GriffithsArticolo} and \cite{GriffithsLibro}, the first part is based on the \emph{infinitesimal variation of Hodge structures} approach to Torelli problems (the Prym \'etale case is described in \cite{SmithVarley}). We prove that, for a general point, the inclusion
\begin{equation}
\label{eq:incldiff}
d\prym\pa{T_\pa{C, \eta, R}} \subset T_{\prym\pa{C, \eta, R}},
\end{equation}
i.e. the position of the image of the differential as a subspace of the tangent space of $\mathcal{A}_{g-1+\frac{r}{2}}^\delta$ at $\prym\pa{C, \eta, R}$, determines
the semi-canonical model $C_\eta$ of $C$, that is the image of the curve $C$ through the embedding associated to $\omega_C\otimes\eta$. Since $C_\eta$ identifies, up to isomorphism, the pair $\pa{C, \eta}$, this proves the theorem when $h^0(\eta^2)=1$, i.e. when $g\geq r$. As in the \'etale case (cf. \cite{DebarreTorelliPrym}), the kernel of the co-differential of the Prym map at $\pa{C, \eta, R}$ is the space of the quadrics vanishing on $C_\eta$. We show that we can recover $C_\eta$ in the intersection of the quadrics. It turns out that the more interesting and difficult case is when $g=6$ and $r=2$. In this case the intersection of the quadrics is, scheme theoretically, the curve $C_\eta$ and five projective lines corresponding to the five morphisms of degree $4$ from $C$ to $\mathbb{P}^1$.

In the second part, using degeneration methods (for the \'etale case see \cite{Beauv77S, friedsmithquadrics}), we prove the theorem for $g<r$. We fix a general curve $C$ and we show that the Prym map is generically injective on the space of double coverings of $C$. To do this we extend the Prym map to some admissible coverings of $C$ (see, for example, Figure \ref{disegno} of page \pageref{disegno}) and we get a proper morphism to a suitable compactification of $\mathcal{A}_{g-1+\frac{r}{2}}^\delta$. Then we prove that the infinitesimal Torelli theorem holds also at the boundary and we give a lemma that allows us to compute the degree of the Prym map once we know the behaviour at the boundary.

\medskip

We give then two applications. By Theorem \ref{theorem: globaltorelli}, a general Prym variety $P$ arises from a unique covering $\pi\colon D \ra C$. Thus there is a canonical choice of a divisor $\Theta_P$ inducing the polarization. In analogy with the hyperelliptic case of Andreotti's proof of the Torelli theorem (see \cite{andreotti}), we prove that the Gauss map of $\Theta_P$ in $P$ allows us to determine the branch locus of $\pi$.

The second application is the following.
\begin{theorem}
\label{theorem:isogjacob}
A very general Prym variety of dimension at least $4$ is not isogenous to a Jacobian variety.
\end{theorem}
We recall that a Prym variety is said to be very general when it is outside a countable union of proper subvarieties of the Prym locus. The argument is based again on degenerations techniques and it is similar to those used in \cite{artPirola} and \cite{NarPir}. However, the comparison of extension classes requires a more sophisticated geometric analysis. In particular, we use that a very general Prym variety of a ramified covering is simple (see \cite{HodgePrym} and \cite{baseN}).

In \cite{mio} the first author proves that on a very general Jacobian variety of dimension $n\geq 4$ there are no curves of genus $g$, with $n<g<2n-2$, and, by using Theorem \ref{theorem:isogjacob}, that there are strong obstructions to the existence of curves of genus $2n-2$. For example, there are no curves of genus $6$ on a very general Jacobian variety of dimension $4$. This was our original motivation.

\subsection{Plan of the paper}

In Section \ref{sec:firsttorelli} and \ref{sec:genericTorelli} we prove Theorem \ref{theorem: globaltorelli}. In Section \ref{subsec:finitness} we prove an infinitesimal version of the Torelli theorem and we determine when the Prym map is finite. In Section \ref{subsec:quadrics} we describe the intersection of the quadrics vanishing on the semi-canonical model $C_\eta$ of $C$. In particular, we show that $C_\eta$ is the only non-hyperplane curve in the intersection. This concludes the proof of the theorem when $g\geq r$.

In Section \ref{subsec:Prymmapboundary} we fix a curve $C$ and we show that it is possible to define a rational map
\[
\mathcal{S}\colon \Upsilon \dashrightarrow  \bar{\mathcal{A}}_{g-1+\frac{r}{2}}^\delta,
\]
from the space of admissible coverings of $C$ to a suitable compactification of the moduli space of abelian varieties, which coincides with the Prym map on the space of smooth coverings. In Section \ref{subsec:blowup} we blow up $\mathcal{S}$ in its indeterminacy locus to get a proper map. Then we give a result on the cardinality of the general fibre of a proper morphism and we apply it to our case in order to show that $\mathcal{S}$ is generically injective (if $g>2$).

In Section \ref{subsec:Gauss} we describe the Gauss map of a Prym variety, while in Section \ref{subsec:prymjac} we prove Theorem \ref{theorem:isogjacob}.

\subsection{Notation and preliminaries}
\label{subsec: preliminaries}

\begin{enumerate}
 \item We work over the field $\mathbb{C}$ of complex numbers. 
Each time we have a family of objects parametrized by a scheme $X$ (respectively by a subset $Y\subset X$) we say that the \emph{general} element of the family has a certain property $\mathfrak{p}$ if $\mathfrak{p}$ holds on a dense Zariski open subset of $X$ (respectively of $Y$). Moreover, we say that a \emph{very general} element of $X$ (respectively of $Y$) has the property $\mathfrak{p}$ if $\mathfrak{p}$ holds on a dense subset that is the complement of a union of countably many proper subvarieties of $X$ (respectively of $Y$).

\item \label{item:Laz} Given an effective line bundle $L$ on $C$, we will often need to consider the natural map
\[
m\colon\sym^2H^0\pa{C,L} \ra H^0\pa{C,L^2}
\]
and its dual
\[
m^*\colon H^1\pa{C,\omega_C\otimes L^{-2}} \ra \sym^2H^1\pa{C,\omega_C\otimes L^{-1}}.
\]
Let $f\colon C \ra \mathbb{P}^N$ be the projective morphism associated to the linear system $\abs{L}$. We recall (cf. \cite{GreenLaz, LazLect}) that the kernel of $m$ can be identified with the space of the homogeneous quadratic polynomials vanishing on $f\pa{C}$. Furthermore, $m^*$ can be identified (up to multiplication by a non-zero scalar) with the map
\[
\ext^1\pa{L, \omega_C\otimes L^{-1}}\ra \Hom\pa{H^0\pa{C,L}, H^1\pa{C, \omega_C\otimes L^{-1}}}
\]
which sends an extension of $L$ by $\omega_C\otimes L^{-1}$ to the connecting homomorphism it determines.

\item \label{item:AbelPrymPrelim} Let $P$ be the Prym variety of a ramified covering $\pi\colon D \ra C$ and let $i\colon D \ra D$ be the involution induced on the curve $D$ by $\pi$. The \emph{Abel-Prym map} of $\pi$ is the morphism
\begin{align*}
a\colon D &\ra J\pa{D}\\
x &\mapsto [x-i\pa{x}].
\end{align*}
The image of $a$ is a curve $D'$ contained in $P$ and it is called \emph{Abel-Prym curve}. If $D$ is hyperelliptic, $a$ has degree two, otherwise it is birational.

\item Given an abelian variety $A$, its \emph{Kummer variety} $\mathcal{K}\pa{A}$ is the quotient of $A$ by the involution that maps $x$ in $-x$. We denote by  $\mathcal{K}^0\pa{A}$ the Kummer variety of $\pic^0\pa{A}$.

\item A \emph{semi-abelian variety} $G$ of rank $n$ is  an extension $0 \ra T \ra G \ra B \ra 0$ of an abelian variety $B$ by an algebraic torus $T=\prod^n \mathbb{G}_m$ (see \cite[Chapter II, Section 2]{ChaiLibro} and \cite{FaltingsLibro}).

\item Let $D$ be a projective curve having only nodes (ordinary double points) as singularities. The \emph{generalized Jacobian variety} of $D$ is defined as the semi-abelian variety $\pic^0\pa{D}$. If $D$ is obtained from a non-singular curve $C$ by identifying $p,q\in C$, the semi-abelian variety $J\pa{D}$ is the extension of $J\pa{C}$ by $ \mathbb{G}_m$ determined by $\pm [p-q]\in \mathcal{K}^0\pa{J\pa{C}}$ (see \cite{SerreLibroJacobDeg}).

\item \label{item:diff} Given a smooth curve $C$ of genus greater than 1, we denote by $\Gamma_C$ the image of the difference map
\begin{align}
C\times C &\ra J\pa{C} \isom \pic^0\pa{J\pa{C}}\\
\notag \pa{p,q} &\mapsto [a-b].
\end{align}
The image $\Gamma'_C$ of $\Gamma_C$ through the projection $\sigma_{C}\colon \pic^0\pa{J\pa{C}} \ra \mathcal{K}^0\pa{J\pa{C}}$ is a surface, in the Kummer variety, that parametrizes generalized Jacobian varieties of rank $1$ with compact part $J\pa{C}$. Given an integer $n\in \mathbb{N}$, we denote by $n\Gamma_C$ the image of $\Gamma_C$ under the multiplication by $n$ and we denote by  $n\Gamma'_C$ the projection of $n\Gamma_C$ in $\mathcal{K}^0\pa{J\pa{C}}$.
\end{enumerate}

\section{A first Torelli-type theorem}
\label{sec:firsttorelli}

\subsection{Finiteness of the Prym map}
\label{subsec:finitness}

In the present section we want to determine when the Prym map is finite. Here we state a simple lemma that we will use throughout the paper.

\begin{lemma}
\label{lemma:tecnico}
Let $\pa{C, \eta, R}$ be a general point of $\mathcal{R}_{g,r}$.
\begin{enumerate}
\item \label{item:tecnico1}If $r=6$ and $g=3$, then $h^0\pa{\eta}=1$.
\item \label{item:tecnico2}
    If one of the following conditions holds
    \begin{enumerate}
        \item  $r=2$  and $g\geq 5$,
        \item  $r=4$  and $g\geq 3$,
        \item $r=6$  and $g\geq 4$,
    \end{enumerate}
    then $\eta$ is not effective.
\item If one of the following conditions holds
       \begin{enumerate}
        \item \label{item:tecnico3a} $r=2$, and $g\geq 4$,
        \item $r=4$, and $g\geq 5$,
    \end{enumerate}
then, given any point $p\in C$, the line bundle $\eta\otimes\fascio{C}\pa{p}$ is not effective.
\item \label{item:tecnico4} If $r=2$ and $g\geq 6$, then, given any two points  $p,q\in C$, the line bundle $\eta\otimes\fascio{C}\pa{p+q}$ is not effective.
\end{enumerate}
\begin{proof}
Statements \eqref{item:tecnico1} and  \eqref{item:tecnico2} are trivial. We prove \eqref{item:tecnico3a}, the other cases being analogous. Assume by contradiction that for each $\eta$ there exists a point $p\in C$ such that $\eta\otimes\fascio{C}\pa{p}\simeq \fascio{C}\pa{a_1+a_2}$ for some $a_1, a_2 \in C$. Then $R+2p\equiv 2a_1+2a_2$ and so $C$ has a $2$-dimensional family of $g^1_4$ of type $\abs{2a_1+2a_2}$. Since the moduli space of coverings of degree $4$ of $\mathbb{P}^1$ with $2$ ramification points over the same branch point has dimension $2g+2$, a general curve of genus $g\geq 4$ has, at most, a $1$-dimensional family of $g^1_4$ of this type and we get a contradiction.
\end{proof}
\end{lemma}

We now give conditions under which the Prym map is a local embedding. We recall that the tangent space of $\mathcal{R}_{g,r}$ at $\pa{C, \eta, R}$ can be identified with the space of first-order deformations of the $r$-pointed curve $C$, i.e. with $H^1\pa{C, T_C\pa{-R}}$ (\cite[Chapter 3, Section B]{HM98}). Following the computation in \cite[Proposition 7.5]{Beauv77} (see also \cite[Proposition 4.1]{LangeOrtega} and \cite[Proposition 3.1]{pol1122} for the ramified case), we can identify the co-differential of the Prym map at $\pa{C, \eta, R}$ with the natural map
\begin{equation*}
d\prym^*\colon \sym^2H^0\pa{C, \omega_C\otimes \eta} \ra H^0\pa{C, \omega_C^2\otimes\fascio{C}\pa{R}}.
\end{equation*}

By \cite[Theorem 1]{GreenLaz} (see also \cite[Theorem 1.1]{LangeSernesi}), in order to prove the surjectivity of $d\prym^*$, it is sufficient to show that $\omega_C\otimes \eta$ is very ample and that the following inequality holds
\[
 \deg\pa{\omega_C\otimes\eta}\geq 2g+1-\cliff\pa{C}.
\]
By a straightforward computation we get the following proposition (cf. \cite[Section 5]{LangeOrtega}):

\begin{prop}
\label{prop:proplocaltorelli}
Let $\pa{C, \eta, R}\in \mathcal{R}_{g,r}$ and assume that one of the following conditions holds
\begin{enumerate}
  \item $r\geq 6$ and $g\geq 1$;
  \item $r=4$, $C$ is not hyperelliptic and $h^0\pa{\eta}=0$;
  \item $r=2$, $\cliff\pa{C}>1$ and $h^0\pa{\eta\otimes\fascio{C}\pa{p}}=0$ for each $p\in C$;
\end{enumerate}
then the differential of the Prym map at $\pa{C, \eta, R}$ is injective.
\end{prop}

In view of Lemma \ref{lemma:tecnico} we can restate Proposition \ref{prop:proplocaltorelli} in a more compact form.

\begin{cor}
\label{cor:Prymfinite}
The Prym map is generically finite (onto its image) if and only if $r\geq 6$ and $g\geq 1$, or $r=4$ and $g\geq 3$, or $r=2$ and $g\geq 5$.
\end{cor}
\begin{proof}
Lemma \ref{lemma:tecnico} proves that the statement of Proposition \ref{prop:proplocaltorelli} is realized for a general point $\pa{C, \eta, R}$. Thus the \emph{if} part is proved.

For the \emph{only if} direction, we notice that if $r=4$ and $g<3$, or $r=2$ and $g< 5$, then $\dim \mathcal{R}_{g,r}> \mathcal{A}_{g-1+\frac{r}{2}}^{\delta}$.
\end{proof}

\subsection{Quadrics vanishing on the semi-canonical model}
\label{subsec:quadrics}

We start with the following definition.

\begin{defin}
Let $\pa{C, \eta, R}$ be a point of $\mathcal{R}_{g,r}$. The projective map
\[
f_\eta\colon C \ra \mathbb{P}H^0\pa{C, \omega_C\otimes \eta}^*,
\]
defined by the linear system $\abs{\omega_C\otimes \eta}$, is called  \emph{semi-canonical map} and $C_\eta:=f_\eta\pa{C}$ is the \emph{semi-canonical model} of $C$ (cf. \cite{DebarreTorelliPrym}).
\end{defin}

In the present section we consider a point $\pa{C, \eta, R}\in \mathcal{R}_{g,r}$ that satisfies the hypotheses of Proposition \ref{prop:proplocaltorelli}. In particular $\omega_C\otimes \eta$ is very ample and $f_\eta$ is an embedding. We recall that the co-differential of the Prym map at $\pa{C, \eta, R}$ is the surjective map
\[
d\prym^{*}\colon \sym^2H^0\pa{C, \omega_C\otimes \eta}\ra H^0\pa{C, \omega_C^2\otimes\fascio{C}\pa{R}}.
\]
As we have remarked in \eqref{item:Laz} of Section \ref{subsec: preliminaries}, $I_{2}\pa{C}:=\ker d\prym^{*}$ is the space of the homogeneous quadratic polynomials vanishing on $C_\eta$, thus the image of $d\prym$ in $T_{\prym\pa{C, \eta, R}}$ determines the projective space $\mathbb{P}I_2\pa{C_\eta}$ of the quadrics containing $C_\eta$ (see \eqref{eq:incldiff} of Section \ref{intr}). The aim is to recover the curve $C_\eta$ in the intersection of the quadrics. We remark that, by a dimensional count, in the following cases
\[
r=6 \text{ and } g=2, \qquad r=4 \text{ and } g=4, \qquad r=2 \text{ and } g=5
\]
this is not possible.

We recall that there is a natural bijective correspondence  between the points of $\mathbb{P}H^0\pa{C, \omega_C\otimes \eta}^*$ lying in the intersection of the quadrics of $\mathbb{P}I_2\pa{C_\eta}$ and the extensions (up to isomorphism and multiplication by a scalar)
\begin{equation*}
0\ra \eta^{-1} \ra F \ra \omega_C\otimes \eta \ra 0
\end{equation*}
with connecting homomorphism $\delta\colon H^0\pa{C,\omega_C\otimes \eta} \ra H^1\pa{C, \eta^{-1}}$ of rank $1$ (see \cite[Lemma 1.2]{LangeSernesi}). Given $p\in \mathbb{P}H^0\pa{C, \omega_C\otimes \eta}^*$ in the intersection of the quadrics, the image of the corresponding $\delta$ in $H^1\pa{C,\eta^{-1}}= H^0\pa{C, \omega_C\otimes \eta}^*$ is the $1$-dimensional vector space defining $p$.

In the proof of Theorem \ref{theorem:quadriche} we will classify all the extensions of the previous type. In order to do this, we will need the following technical results on the number of points of order two lying on the theta divisor of a Jacobian. For a more exhaustive analysis of this topic we refer to \cite{miatesi}.

\begin{prop}
\label{prop:ordinedue}
Let $C$ be a general curve of genus $g$, then the followings hold.
 \begin{enumerate}
  \item \label{item:punti di ordine 2}Let $\Theta$ be a symmetric theta divisor in $J\pa{C}$. For each $a \in A$ there exists a point of order two not contained in $t_a^*\Theta$.
  \item \label{item:ptiord2perapplicazione} Let $M_1,\ldots M_N$ be a finite number of line bundles of degree $d\leq g-1$ on $C$. Given an integer $k \leq g-1-d$, if $\eta$ is a general line bundle of degree $k$ such that $h^0\pa{\eta^2}>0$, then
\[
h^0\pa{\eta \otimes M_i}=0 \qquad \forall i=1,\ldots, N.
\]
 \end{enumerate}
\end{prop}
\begin{proof}
\begin{enumerate}
Statement \eqref{item:punti di ordine 2} can be proved by reducing $J\pa{C}$ to the product of elliptic curves. In this case $t_a^*\Theta$ contains at most $2^{2g}-3^g$ points of order two.

We will prove statement \eqref{item:ptiord2perapplicazione} for $N=1$. The general case follows immediately.

Let $\Lambda:=\set{\eta \in \pic^{k}\pa{C}: h^0\pa{\eta^2}>0}$, we prove that $\Lambda':=\set{\eta \in \Lambda: h^0\pa{M_1\otimes\eta}>0}$ is a proper subset. Namely we claim that, given  an effective $L\in \pic^{2k}\pa{C}$, there exists a line bundle $\eta\in \Lambda \setminus \Lambda'$ such that $\eta^2\simeq L$. Setting $n:=g-1-k-d$ and considering the line bundle $M_1^2 \otimes L \otimes \fascio{C}\pa{p}^{2n}$, where $p$ is an arbitrary point of $C$, we can assume $M_1\simeq \fascio{C}$ and $k=g-1$. Then the result follows from \eqref{item:punti di ordine 2}.
\end{enumerate}
\end{proof}

\begin{rmk}
Statement \eqref{item:punti di ordine 2} of Proposition \ref{prop:ordinedue} can be further improved. Namely, let $A$ be a principally polarized abelian variety of dimension $g$ and let $\Theta$ be its a symmetric theta divisor. For each $a \in A$ there are at most $2^{2g}-2^{g}$ points of order two lying on $t_a^*\Theta$. (see \cite[Proposition 2.21]{miatesi}) If $\Theta$ is irreducible and $t_a^*\Theta$ is not symmetric with respect to the origin, the statement also holds for $2^{2g}-\pa{g+1}2^{g}$ points. One might expect the right bound to be $2^{2g}-3^g$ as in the product of elliptic curves.
\end{rmk}

In the proof of Theorem \ref{theorem:quadriche} we will also use the following lemma (see for instance \cite[Lemma X.7]{BeauvilleBook}).

\begin{lemma}
\label{lemma:subline}
Let $C$ be a smooth projective curve and let $F$ be a vector bundle of rank $2$ on $C$. If
\begin{equation}
\label{eq:sezvectorbundle}
  2h^0\pa{F}-3 > h^0\pa{\det F},
\end{equation}
then there exists a line bundle $L\subset F$ such that $h^0\pa{L}\geq 2$ and $F/L$ is a line bundle.
\end{lemma}

\begin{theorem}
\label{theorem:quadriche}
Let $\pa{C, \eta, R} \in \mathcal{R}_{g,r}$ and let $\mathbb{P}I_2\pa{C_\eta}$ be the space of the quadrics vanishing on the semi-canonical model $C_\eta$ of $C$.
\begin{enumerate}
\item \label{eq:itemquadr1}In the following cases:
\begin{itemize}
  \item $r\geq 8$ and $g\geq 1$;
  \item  $r=6$, $C$ is not hyperelliptic and $h^0\pa{\eta}=0$;
  \item  $r=4$, $\cliff\pa{C}>1$ and $h^0\pa{\eta\otimes\fascio{C}\pa{p}}=0$ for each $p\in C$;
  \item  $r=2$, $\cliff\pa{C}>2$ and $h^0\pa{\eta\otimes\fascio{C}\pa{p+q}}=0$ for each $p,q\in C$.
\end{itemize}
The intersection of the quadrics of $\mathbb{P}I_2\pa{C_\eta}$ is, set-theoretically, $C_\eta$.
\item \label{eq:itemquadr2} If $r=6$, $g=3$, $C$ is not hyperelliptic and $h^0\pa{\eta}=1$, the intersection of the quadrics of $\mathbb{P}I_2\pa{C_\eta}$ consists of the curve $C_\eta$ and a line (note that $C_\eta$ is not a plane curve).
\item \label{eq:itemquadr3} If $r=2$, $g=6$, and $\pa{C, \eta, R}$ is a general point of $\mathcal{R}_{g,r}$, the intersection of the quadrics of $\mathbb{P}I_2\pa{C_\eta}$ consists of the curve $C_\eta$ and of five projective lines.
\end{enumerate}
\end{theorem}
\begin{proof}
In view of the previous discussion, we want to classify the extensions of type
\begin{equation}
\label{eq: extgenT}
  0 \ra \eta^{-1} \ra F \ra  \omega_C\otimes \eta \ra 0
\end{equation}
with connecting homomorphism of rank $1$. Since, by hypothesis, $g\geq 7-r$, we can apply Lemma \ref{lemma:subline}: $F$ has a sub-line bundle $L$ such that $h^0\pa{L}\geq 2$ and $F/L$ is a line bundle. Fixing such a sub-line bundle the situation is summarized in the following diagram, in which \eqref{eq: extgenT} appears as the horizontal sequence.
\footnotesize
\begin{equation}
\label{eq:diagL}
\xymatrix{
&&0\ar[d]&&\\
&&L \ar@{-->}[dr]^{\tau}\ar[d]&&\\
0\ar[r]&\eta^{-1}\ar[r]&F\ar[d]\ar[r]&\omega_C\otimes \eta \ar[r]&0\\
&&\omega_C\otimes L^{-1}\ar[d]&&\\
&&0&&\\
}
\end{equation}
\normalsize
We notice that, since the map $\tau$ of diagram \eqref{eq:diagL} is not zero, it holds that:
\begin{equation}
\label{eq:taunotzero}
h^0\pa{\omega_C\otimes \eta \otimes L^{-1}}>0.
\end{equation}
If $L$ is not special, $h^0\pa{L}=h^0\pa{F}=h^0\pa{\omega_C\otimes \eta}-1$ and so $L=\omega_C\otimes \eta\otimes\fascio{C}\pa{-p}$ for some $p\in C$. By diagram \eqref{eq:diagL} we get
\begin{equation*}
\begin{split}
 h^0\pa{L} & +h^0\pa{\omega_C \otimes L^{-1}}\geq h^0\pa{F},\\
 h^0\pa{L} & -h^0\pa{\omega_C\otimes L^{-1}}= \deg L+1-g.
 \end{split}
 \end{equation*}
Thus, when $L$ is special, we obtain
\begin{equation}
\label{eq:cliff2}
3-\frac{r}{2} \geq \cliff\pa{L} \geq 0.
\end{equation}

\textbf{Proof of \eqref{eq:itemquadr1}:} In the hypotheses of \eqref{eq:itemquadr1}, equations \eqref{eq:taunotzero} and \eqref{eq:cliff2} are never simultaneously satisfied. Thus $L$ is not special and $C_\eta$ is, set-theoretically, the intersection of quadrics.

\textbf{Proof of \eqref{eq:itemquadr2}:} By  \eqref{eq:cliff2} and \eqref{eq:taunotzero}, we get $L=\omega_C$ and $F=\omega_C\oplus \fascio{C}$. Furthermore, by hypothesis, there is only a non-zero morphism $\omega_C \ra \omega_C\otimes \eta$. Therefore there is a $2$ dimensional family of extensions of type
\[
0\ra \eta^{-1}\ra \omega_C \oplus \fascio{C} \ra \omega_C\otimes \eta \ra 0
\]
corresponding to the points of the line determined by
\[
H^0\pa{\omega_C}\subset H^0\pa{\omega_C\otimes \eta}.
\]

\textbf{Proof of \eqref{eq:itemquadr3}:} The proof is long and it will be divided into $4$ steps. We recall at first that on a general curve $C$ of genus $6$ there are exactly $5$ non-isomorphic line bundles $M_1,\ldots ,M_5$ of degree $4$ such that $h^0\pa{M_i}=2$ (\cite[Chapter 5]{acgh}).

\emph{Step I: $L \simeq M_i$ for some $i$.}

\noindent By \eqref{eq:cliff2} we have $\cliff\pa{L}=2$. Since $C$ is a general curve of genus $6$, either $\deg L=4$ and $h^0\pa{L}=2$, or $\deg L=6$ and $h^0\pa{L}=3$, or $\deg L=8$ and $h^0\pa{L}=4$. In the last case $L\simeq \omega_C\otimes \fascio{C}\pa{-p-q}$ for some $p,q\in C$, but this contradicts \eqref{item:tecnico4} of Lemma \ref{lemma:tecnico}.  It follows that either $L\simeq M_i$ or $L\simeq \omega_C\otimes M_i^{-1}$ for some $i=1,\ldots, 5$. By \eqref{item:ptiord2perapplicazione} of Proposition \ref{prop:ordinedue},
\begin{equation}
\label{eq:noneffect}
h^0\pa{M_i\otimes\eta}=0 \qquad \forall i=1, \ldots, 5,
\end{equation}
for a  general $\eta$. Then there are no non-zero maps from $\omega_C\otimes M_i^{-1}$ to $\omega_C\otimes \eta$ and we can assume $L\simeq M_i$.

\emph{Step II: the vector bundle $F$ is determined by $M_i\subset F$.}

\noindent Notice that
\[
h^0\pa{M_i}+h^0\pa{\omega_C\otimes M_i^{-1}}=5=h^0\pa{F}.
\]
It follows that the vertical exact sequence of diagram \eqref{eq:diagL} (see Proposition \ref{prop:proplocaltorelli}), that is
\begin{equation}
\label{eq:exseqTorelli}
0\ra M_i \ra F \ra \omega_C\otimes M_i^{-1}\ra 0,
\end{equation}
is exact on the global sections. Furthermore, by \eqref{eq:noneffect}, it cannot split. The dimension of the space of the extensions of this type is the dimension of the co-kernel of the map
\[
m \colon \sym^2H^0\pa{C, \omega_C\otimes M_i^{-1}}\ra  H^0\pa{C, \omega_C^{2} \otimes M_i^{-2}}.
\]
The image of $C$ under the morphism associated to $\abs{\omega_C\otimes M_i^{-1}}$ is a plane curve of degree $6$. Therefore $m$ is injective and $\dim H^0\pa{C, \omega_C^{2} \otimes M_i^{-2}}- \rk m =1$.

\emph{Step III: the intersection of the quadrics of $\mathbb{P}I_2\pa{C_\eta}$ is, set-theoretically, contained in the union of $C_\eta$ with $5$ projective lines.}

\noindent Dualizing and tensoring \eqref{eq:exseqTorelli} with $\omega_C\otimes \eta$, we get
\[
0\ra M_i \otimes \eta \ra F^*\otimes \omega_C \otimes \eta \ra   \omega_C \otimes M_i^{-1} \otimes \eta  \ra 0.
\]
By \eqref{eq:noneffect}, $h^1\pa{C, M_i\otimes\eta}=0$ and
\[
h^0\pa{C,F^*\otimes\omega_C\otimes \eta}=h^0\pa{C, \omega_C\otimes M_i^{-1} \otimes \eta}=7+1-6+h^0\pa{M_i\otimes \eta^{-1}}=2.
\]
It follows that there is a two dimensional family of maps $F \ra \omega_C\otimes \eta$.

\emph{Step IV: the intersection of the quadrics of $\mathbb{P}I_2\pa{C_\eta}$ coincides, as a scheme, with the union of $C_\eta$ with $5$ projective lines.}

\noindent We recall that
\[
\dim I_2\pa{C_\eta}=4 \qquad \text{and}\qquad \dim\mathbb{P}H^0\pa{C, \omega_C\otimes \eta}=5,
\]
thus the intersection of the quadrics is complete and it is consequently Cohen-Macaulay (\cite[Chapter II, Proposition 8.23]{hart}). It follows that it has degree $16$ (\cite[Chapter I, Theorem 7.7]{hart}) and that it has no embedded components (\cite[Chapter 6, Section 16, Theorem 30]{matsumura}). Since the degree of $C_\eta$ is $11$, this concludes the proof.
\end{proof}

\begin{rmk}
Cases \eqref{eq:itemquadr1} and \eqref{eq:itemquadr2} of Theorem \ref{theorem:quadriche} can be also deduced from \cite[Section 1, Proposition]{SD72} and \cite[Theorem 1.3]{LangeSernesi}.
\end{rmk}

\begin{cor}
\label{cor:gentorellifirst}
Let $\pa{C, \eta, R}$  be a general point of $\mathcal{R}_{g,r}$. If one of the following conditions holds
\[
r>6\text{ and }g\geq 1,\qquad r=6\text{ and }g\geq 3,\qquad r=4\text{ and }g\geq 5,\qquad r=2\text{ and }g\geq 6,
\]
the image of the differential of the Prym map at $\pa{C, \eta, R}$ determines the pair $\pa{C, \eta}$.
\begin{proof}
Lemma \ref{lemma:tecnico} proves that the statement of Theorem \ref{theorem:quadriche} is realized for a general point $\pa{C, \eta, R}$.
\end{proof}
\end{cor}

\section{The generic Torelli theorem}
\label{sec:genericTorelli}

\subsection{The Prym map at the boundary}
\label{subsec:Prymmapboundary}

In the following we assume $C$ to be a general curve of genus $1<g<r$, with $r\geq 6$. Set
\begin{equation}
\label{eq:superf}
\Upsilon:=\set{\pa{\eta, R} \in \pic^{\frac{r}{2}}\pa{C} \times C_r: R\in \abs{\eta^2}},
\end{equation}
and consider the partition
\[
\Upsilon=\bigsqcup_{k=1}^r Y_k,
\]
where
\begin{equation}
\label{eq:defD_k}
Y_{k}:=\set{\pa{\eta, \sum_{i} n_iy_i} \in \Upsilon : \sum_i\pa{n_i-1}=k-1}.
\end{equation}
We remark that $\Upsilon$ is an \'etale $2^{2g}$-covering of the symmetric product $C_r$ of $C$ and, in particular, $Y_1$ is an \'etale covering of the open set of divisors with no multiple points.
The rational map
\begin{equation}
\begin{split}
\label{eq:razriv}
\mathcal{T} \colon \Upsilon &\dashrightarrow \mathcal{R}_{g,r}\\
\pa{\eta, R}&\mapsto \pa{C, \eta, R},
\end{split}
\end{equation}
is clearly defined over $Y_1$ and, for $g>2$, here it is an isomorphism. If $g=2$, the map $\mathcal{T}|_{Y_1}$ has degree two: namely, if $i\colon C \ra C$ is the hyperelliptic involution, then 
\begin{equation}
 \label{eq:gradoinv}
\mathcal{T}^{-1}\pa{C, \eta, R}=\set{\pa{\eta, R}, \pa{i^*\eta, i\pa{R}}}.
\end{equation}
Let
\begin{equation}
\label{eq:prym'}
  \prym'\colon  \Upsilon \dashrightarrow \mathcal{A}^{\delta}_{g-1+\frac{r}{2}}
\end{equation}
be the composition (regular over $Y_1$) of $\mathcal{T}$ with the Prym map. In this section we extend the rational map $\prym'$.

Let $\pa{\eta,R}$ be a point of $Y_k$ and $\varphi \colon \Delta \ra \Upsilon$ be a non constant map from the complex unit disk to $\Upsilon$ such that $\varphi\pa{0}=\pa{\eta, R}$ and $\varphi\pa{\Delta \setminus \set{0}}\subset Y_1$. By pullback, up to a finite base change, we have a map of families of curves
\begin{equation}
\label{eq:famD_kcompl}
\xymatrix{
\mathcal{D}\ar[rr]^{\Gamma} \ar[rd] && \mathcal{C} \ar[dl]\\
&\Delta &
}
\end{equation}
with the following properties:
\begin{itemize}
  \item for $t\neq 0$, the curve $\mathcal{C}_t$ is isomorphic to $C$ and $\Gamma_t \colon \mathcal{D}_t\ra \mathcal{C}_t$ is the double covering of smooth projective curves corresponding to $\pa{C, \varphi\pa{t}}\in \mathcal{R}_{g, r}$;
  \item $\mathcal{D}_0 \ra \mathcal{C}_0$ is a double admissible covering of semi-stable curves (see \cite{HarrisMumford} and \cite[Chapter 3, Section G]{HM98}).
\end{itemize}

The family of coverings in \eqref{eq:famD_kcompl} determines a family of semi-abelian varieties
\begin{equation}
\label{eq:famP}
\mathcal{P}/\Delta,
\end{equation}
where $\mathcal{P}_t$ is, for each $t\in \Delta \setminus \set{0}$, the Prym variety of the double covering $\mathcal{D}_t\ra \mathcal{C}_t$, while $\mathcal{P}_0$ is the kernel of the morphism of semi-abelian varieties 
\[
J\pa{\mathcal{D}_0}\ra J\pa{\mathcal{C}_0}.
\]

\begin{prop}
\label{prop:limP0}
Let $\varphi \colon \Delta \ra \Upsilon$ be a non constant map from the complex unit disk to $\Upsilon$ such that $\varphi\pa{0}=\pa{\eta, R}\in Y_k$ and $\varphi\pa{\Delta \setminus \set{0}}\subset Y_1$. Consider the family of semi-abelian varieties $\mathcal{P}/\Delta$ defined as in \eqref{eq:famP}.
\begin{enumerate}
\item If $k=2$, the semi-abelian variety  $\mathcal{P}_0$ has rank $1$ and its compact part is irreducible. Furthermore, it is uniquely determined by the point $\pa{\eta, R}$.
\item If $k=3$, $\mathcal{P}_0$ is either a trivial extension of rank $1$, or a product of an abelian variety and an elliptic curve, or a semi-abelian variety of rank $2$.
\item If $k>3$,  $\mathcal{P}_0$  is either a semi-abelian variety of rank $1$ with reducible compact part or a semi-abelian variety of rank greater than $1$.
\end{enumerate}
\end{prop}
\begin{proof}
In the following we want to describe the possible admissible coverings $\mathcal{D}_0 \ra \mathcal{C}_0$. Roughly speaking, the covering is obtained as the limit of a smooth double covering of $C$ when two or more branch points come together. If, for example, $R=R'+kx$, then $k$ branch points on $C$ collapse on the point $x$. In order to get the stable limit of the double covering we have to attach a rational curve to $C$ in $x$. More generally, if $\pa{\eta, R}\in Y_k$, we have $\mathcal{C}_0 \simeq C  \cup F$, where $F$ is a union of (possibly singular and reducible) rational curves intersecting $C$ in the points of the support of $R$ of multiplicity greater than $1$. The curve $\mathcal{D}_0$ is isomorphic to $E\cup G$, where $E$ is a smooth double covering of $C$ and $G$ is a union of (possibly singular and reducible) curves mapping 2 to 1 on $F$ (see \cite{HarrisMumford, HM98}).

We want to describe the admissible covering $\mathcal{D}_0 \ra \mathcal{C}_0$ and the semi-abelian variety $\mathcal{P}_0$ for $k=2, 3$ and $k>3$.
\begin{enumerate}
\item If $k=2$, then $R=R'+2x$ (see Figure \ref{disegno}). The double covering $\pi\colon E \ra C$ has $r-2>0$ branch points. Since $x$ is not a branch point for this covering, $G$ intersects  $E$ in two different points $p_x$ and $q_x$ such that $\pi^{-1}\pa{x}=\{p_x, q_x\}$. We can conclude that $G\simeq \mathbb{P}^1$ and $F\simeq \mathbb{P}^1$.

    The semi-abelian variety $\mathcal{P}_0$ has rank $1$ and it is the extension of the Prym variety $P$ of $\pi$ determined by $\pm[p_x-q_x]\in \mathcal{K}^0\pa{P}$.

\item If $\pa{\eta, R}\in Y_3$, either $R=R'+2x+2y$ or $R=R'+3x$. In the first case, the double covering $\pi\colon E \ra C$ has $r-4>0$ branch points and so $g\pa{E}=2g-3+\frac{r}{2}$. It follows that $\mathcal{P}_0$ is a semi-abelian variety of rank $2$.

In the second case, $\pi \colon E \ra C$ has $r-2$ branch points and $x$ is a branch point for this covering. Therefore $G$ intersects  $E$ in the point $p_x=\pi^{-1}\pa{x}$ and, consequently, the arithmetic genus of $E$ is $1$. There are two possibilities: $E$ is either a rational nodal curve (see Figure \ref{disegno2}) or a smooth elliptic curve. Then $\mathcal{P}_0$ is either a trivial extension of the Prym variety $P$ of $\pi$ by an algebraic torus of rank $1$ or it is the product of $P$ by an elliptic curve.

\item  If $\pa{\eta, R}\in Y_k$ and $k>3$, the situation is more complicated and, as in the case $k=3$, there are different possible limits. We observe only that $E$ has genus strictly lower than $2g-2+\frac{r}{2}$. It follows that either the compact part of $\mathcal{P}_0$ is reducible or $\mathcal{P}_0$ is a semi-abelian variety of rank greater than $1$.
\end{enumerate}
\begin{figure}
\parbox{1.2in}{ \begin{center}
    \begin{picture}(200,200)
     \put(-100,-20){\includegraphics[scale=0.4]{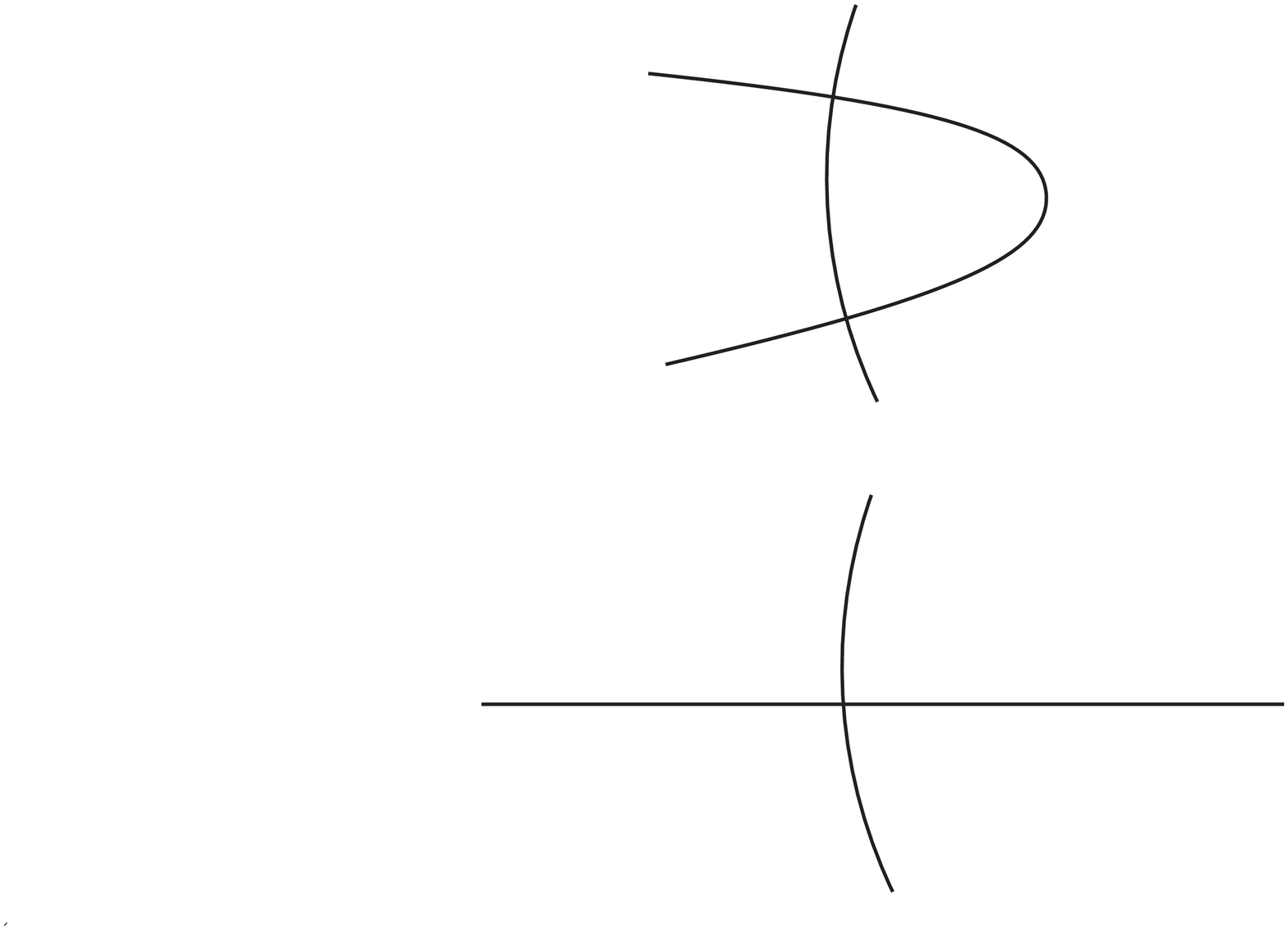}}
     \put(35,40){$C$}
     \put(60,190){$E$}
      \put(110,70){$F\simeq \mathbb{P}^1$}
     \put(105,40){$x$}
      \put(90,130){$q_x$}
      \put(90,190){$p_x$}
      \put(110,200){$G\simeq \mathbb{P}^1$}
  \put(-5,-20){\caption{Admissible double covering ($k=2$)}\label{disegno}}
    \end{picture}
   \end{center}}
\qquad \qquad \qquad
\begin{minipage}{1.2in}
 \begin{center}
    \begin{picture}(600,187)
     \put(70,0){\includegraphics[scale=0.4]{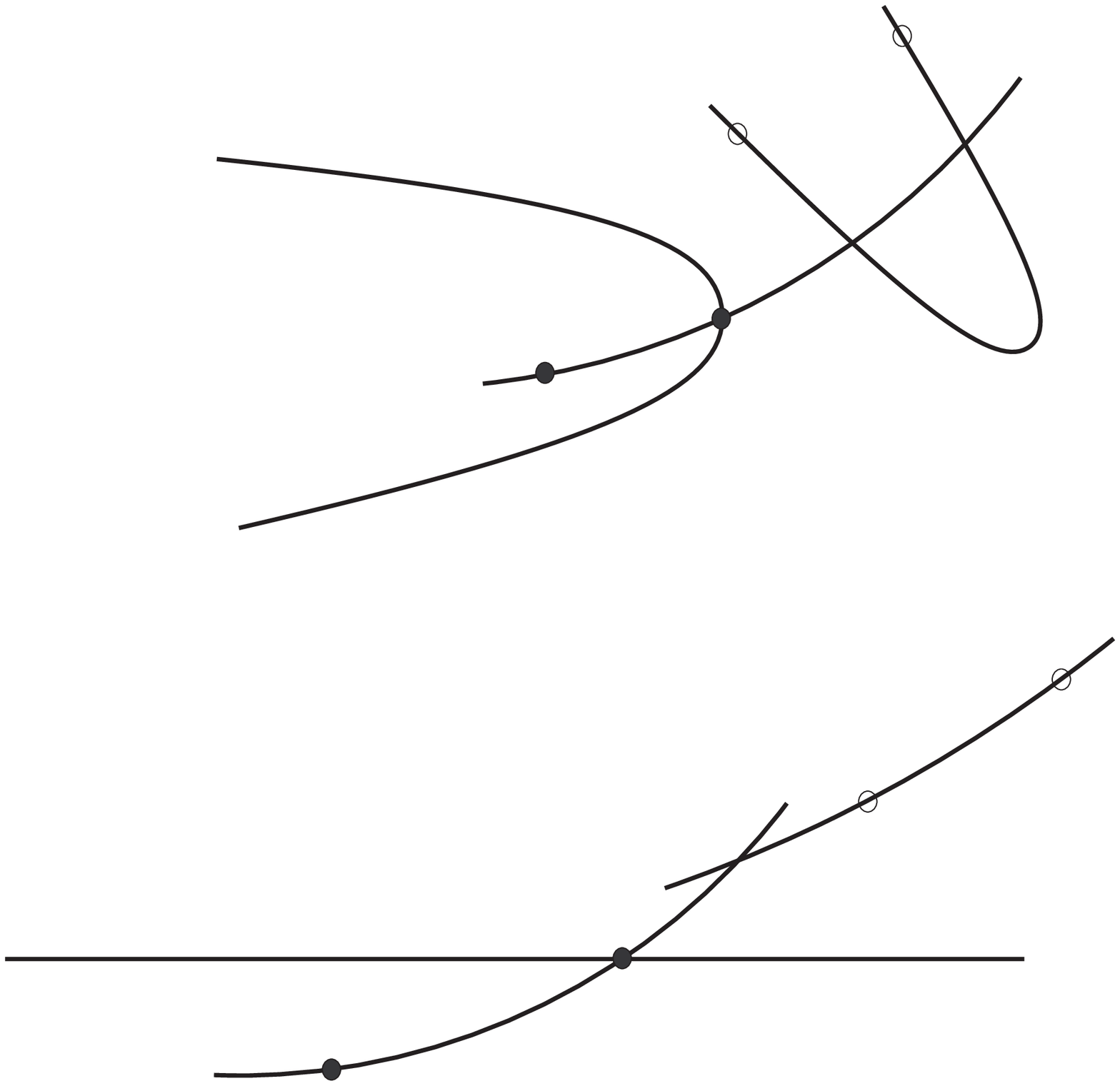}}
     \put(70,35){$C$}
     \put(100,110){$E$}
      \put(155,140){$\mathbb{P}^1$}
      \put(195,190){$\mathbb{P}^1$}
      \put(120,10){$\mathbb{P}^1$}
       \put(190,50){$F$}
       \put(180,30){$x$}
       \put(230,140){$G$}
       \put(210,140){$p_x$}
      \put(280,90){$\mathbb{P}^1$}
  \put(70,-27){\caption{Admissible double covering ($k=3$)}}\label{disegno2}
    \end{picture}
   \end{center}
\end{minipage}
\end{figure}
\end{proof}

Proposition \ref{prop:limP0} suggests us that, in order to extend $\prym'$ (see \eqref{eq:prym'}), we need to compactify the moduli space of abelian varieties $\mathcal{A}_{g-1+\frac{r}{2}}^\delta$. We consider the normalized blowing up of the Satake compactification of $\mathcal{A}_{g-1+\frac{r}{2}}^\delta$ and we denote it by $\bar{\mathcal{A}}_{g-1+\frac{r}{2}}^\delta$ (see \cite{NamikawaI,NamikawaII, NamikawaLibro, Grush}). It is a projective variety and its boundary points parametrize polarized semi-abelian varieties.

From Proposition \ref{prop:limP0} and its proof we have the following corollary.

\begin{cor}
\label{cor:estmappa}
It is possible to extend $\prym'$ to a rational map
\begin{equation*}
\mathcal{S}\colon \Upsilon \dashrightarrow  \bar{\mathcal{A}}_{g-1+\frac{r}{2}}^\delta,
\end{equation*}
such that:
\begin{enumerate}
\item the indeterminacy locus of $\mathcal{S}$ is contained in $\bigsqcup_{k\geq 3}Y_k$;
\item $\mathcal{S}\pa{\eta, R}= \prym \pa{C, \eta, R}$ for $\pa{\eta, R}\in Y_1$;
\item \label{item:estMappa3} given $z=\pa{\eta'\otimes \fascio{C}\pa{x}, R'+2x}$, let $\pi \colon E \ra C$ be the double covering associated to $\pa{C, \eta', R'}$. Then $\mathcal{S}\pa{z}$ is described by the following data:
\begin{itemize}
 \item the compact Prym variety $P$ of $\pi$;
 \item the class $\pm[p_x-q_x] \in \mathcal{K}\pa{P}$, where $\pi^{-1}\pa{x}=\set{p_x,q_x}$.
\end{itemize}
\end{enumerate}
\end{cor}

The next section is devoted to proving the following proposition. This is the last step in completing the proof of Theorem \ref{theorem: globaltorelli}.

\begin{prop}
\label{prop: torelliQ}
Let $C$ be a general curve of genus $1<g<r$ and assume $r\geq 6$.
The rational map
\[
\mathcal{S} \colon \Upsilon \dashrightarrow  \bar{\mathcal{A}}_{g-1+\frac{r}{2}}^\delta,
\]
has generically degree one for $g>2$ and degree two for $g=2$.
\end{prop}

\subsection{Proof of Proposition \ref{prop: torelliQ}}
\label{subsec:blowup}
We want to apply to $\mathcal{S}$ an abstract lemma on the degree of a proper morphism. 

\begin{lemma}
\label{lemma:astratto}
Let $f\colon X \ra Z$ be a generically finite, surjective, proper morphism of varieties over an algebraically closed field $\Bbbk$. Consider an integral, locally closed subset $Y$ of $X$ of codimension $1$, that is not contained in the singular locus of $X$, and set $H:=f\pa{Y}\cap f\pa{Y^c}$. Assume that:
\begin{enumerate}
  \item \label{item:lemmaastratto1} the codimension of the closure $\bar H$ of $H$ in $Z$ is at least $2$;
  \item \label{item:lemmaastratto2} the differential $df$ is injective in a non-empty open set of $Y$;
  \item \label{item:lemmaastratto3} there is a non-empty open set $V$ of $f\pa{Y}$ such that $f^{-1}\pa{y}$ has cardinality $n$ for each $y\in V$.
\end{enumerate}
Then there is a non-empty  open set $U$ of $Z$ such that $f^{-1}\pa{y}$ has cardinality $m\leq n$ for each $y\in U$.
\begin{proof}
We can prove the statement when $Y$ is a closed subset of $X$. Namely, given $Y$ locally closed, conditions \eqref{item:lemmaastratto2} and \eqref{item:lemmaastratto3} clearly also hold for $\bar Y$. In order to prove  \eqref{item:lemmaastratto1}, notice that
\[
f\pa{\bar Y}\cap f\pa{\bar Y^c}\subset H \cup \pa{f\pa{\bar Y \setminus Y}\cap f\pa{Y^c}}\subset H\cup f\pa{\bar Y \setminus Y}.
\]
It follows that
\[
\codim\pa{\ol{f\pa{\bar Y}\cap f\pa{\bar Y^c}}}\geq \min \set{\codim \bar H,\codim f\pa{\bar Y \setminus Y}}\geq 2.
\]

Observe that, up to restriction, we can assume $X$ to be smooth. We claim that we can also assume $f$ to be a finite morphism. Let $W\subset Z$ be the maximal open subset of $Z$ such that $f^{-1}\pa{W}\ra W$ is a proper quasi-finite morphism. By \cite[Theorem 8.11.1]{EGAIV}, $f|_{f^{-1}\pa{W}}$ is a finite morphism. We show that $f\pa{Y}$ is not contained in the complement of $W$ and so, by \eqref{item:lemmaastratto1}, $Y \cap f^{-1}\pa{W}\neq \emptyset$.
This implies that all the hypotheses still hold when we replace $X$ with $f^{-1}\pa{W}$, $Z$ with $W$ and $Y$ with $Y\cap f^{-1}\pa{W}$. Assume, by contradiction that $f\pa{Y}\subset W^c$; then, by \eqref{item:lemmaastratto2}, $f\pa{Y}$ is an irreducible component of $W^c$. Let us consider the union of the closure of $W^c\setminus f\pa{Y}$ and $f\pa{Y} \setminus V$. The complement of this set is open and it has the same property of $W$, thus we get a contradiction.

Let $\tilde Z$ be the normalization of $Z$. By the universal property of normalization we can factorize $f$ as
\[
X \xrightarrow{\ g \ } \tilde Z \xrightarrow{\ \pi \ } Z,
\]
where $\pi$ and $g$ are finite morphisms.
Furthermore there is an open set $T$ of $Z$ such that $T^c$ has codimension at least two and $\tilde T:=\pi^{-1}\pa{T}$ is smooth. This implies that
\[
g \colon g^{-1}\pa{\tilde T} \ra \tilde T
\]
is a finite flat morphism. We recall that $f\pa{Y}$ has codimension $1$ in $Z$, then $Y':=Y\cap g^{-1}\pa{\tilde T}$ is a non-empty set. By \eqref{item:lemmaastratto2}, $Y'$ is not contained in the ramification locus of $g$ and, by \eqref{item:lemmaastratto1}, we can assume that $g^{-1}\pa{g\pa{Y'}}=Y'$. By \eqref{item:lemmaastratto3}, we can conclude that the degree of $g$ is at most $n$. Since $\pi$ is birational, the statement is proved.
\end{proof}
\end{lemma}

Solving the indeterminacy locus (by a suitable blow up) of the rational map
\[
\mathcal{S}\colon \Upsilon \dashrightarrow  \bar{\mathcal{A}}_{g-1+\frac{r}{2}}^\delta,
\]
we obtain a proper map
\[
\mathcal{Q}\colon \mathcal{B}\pa{\Upsilon} \ra \mathcal{Q}\pa{\mathcal{B}\pa{\Upsilon}}\subset \bar{\mathcal{A}}_{g-1+\frac{r}{2}}^\delta
\]
where $\mathcal{B}\pa{\Upsilon}$ has a projection
\[
p \colon \mathcal{B}\pa{\Upsilon} \ra \Upsilon.
\]
Set $X:=\mathcal{B}\pa{\Upsilon}$, $Z:=\mathcal{Q}\pa{\mathcal{B}\pa{\Upsilon}}$, $f:=\mathcal{Q}$ and $Y:=p^{-1}\pa{Y_2}$, where we recall that
\begin{equation*}
Y_{2}:=\set{\pa{\eta, \sum_{i} y_i+ 2x} \in \Upsilon}.
\end{equation*}
By Corollary \ref{cor:Prymfinite} the morphism $\mathcal{Q}$ is generically finite. Moreover, when $g=2$, then $\deg \mathcal{Q}\geq 2$ (see \eqref{eq:gradoinv}). Proposition \ref{prop: torelliQ} will follow from a direct application of Lemma \ref{lemma:astratto}, once we have shown that conditions \eqref{item:lemmaastratto1}, \eqref{item:lemmaastratto2}, \eqref{item:lemmaastratto3} are fulfilled (with $n=1$ when $g>2$ and with $n=2$ when $g=2$).

By Proposition \ref{prop:limP0}, the points in $\mathcal{Q}\pa{Y}\cap \mathcal{Q}\pa{Y^c}$ parametrize semi-abelian varieties of rank $1$ that are trivial extensions. On the other hand, by Corollary \ref{cor:estmappa}, to the general point of $\ol{\mathcal{Q}\pa{Y}}$ is associated an irreducible semi-abelian variety of rank $1$. Therefore $\ol{\mathcal{Q}\pa{Y}\cap \mathcal{Q}\pa{Y^c}}$ is a proper closed subset of $\ol{\mathcal{Q}\pa{Y}}$ and condition \eqref{item:lemmaastratto1} is satisfied. Statements \eqref{item:lemmaastratto2} and \eqref{item:lemmaastratto3} will follow respectively from Lemma \ref{lemma:diff} and Lemma \ref{lemma:genin} (Lemma \ref{lemma:genin2} if $g=2$) below. We recall that, by Corollary \ref{cor:estmappa}, the indeterminacy locus of $\mathcal{S}$ is contained in $\bigsqcup_{k\geq 3}Y_k$. Therefore $\mathcal{Q}|_Y$ coincides with $\mathcal{S}|_{Y_2}$.

\bigskip

In the following lemma we adapt to our case the proof of the infinitesimal Torelli theorem for curves at the boundary (see for instance \cite{usui}).

\begin{lemma}
\label{lemma:diff}
The differential of the map $\mathcal{S}$ at a general point $z\in Y_2$ is injective.
\end{lemma}
\begin{proof}
Let $z$ be a general point in $Y_2$, $z:=\pa{\eta, R}$, $R:=R'+2x$ and $\eta:=\eta'\otimes \fascio{C}\pa{x}$. Moving $x\in C$, we define a one dimensional subvariety of $Y_2$
\[
W:=\set{\pa{\eta'\otimes\fascio{C}\pa{y}, R'+2y}: y\in C\setminus \supp R'}.
\]
Clearly, $W$ is birational to $C$.
The inclusions $W \subset Y_2 \subset \Upsilon$ induce a filtration
\begin{equation}
\label{eq:primafiltr}
T''_z \subset T'_z\subset T_z,
\end{equation}
where $T_z$ is the tangent space of $\Upsilon$ at the point $z$, $T'_z$ is the tangent space of $Y_2$ at $z$, and $T''_z$ is the tangent space of $W$ at $z$. We recall (see \eqref{eq:superf}) that $\Upsilon$ is an \'etale covering of the symmetric product $C_r$ and so we can identify $T_z$ with the tangent space of $C_r$ at $R$. Note that, under this identification, $T'_z$ is the tangent space of the diagonal of $C_r$ passing through $R$ and, consequently, we can identify it with the tangent space of $C_{r-1}$ at the point $R'+x$. We have two exact sequences
\begin{align*}
0 \ra T'_z \ra &T_z \ra N \ra 0,\\
0 \ra T''_z \ra &T'_z \ra N' \ra 0,
\end{align*}
where $N'$ corresponds to the tangent space of $C_{r-2}$ at the point $R'$.

The next step is to describe the tangent space at $\mathcal{S}\pa{z}$. We can put the period matrix $\bar M$ corresponding to the point $\mathcal{S}\pa{z}$ in the form
\begin{equation*}
\label{eq:periodmatrix}
\bar M= \left(
  \begin{array}{cc}
    0 & w^t \\
    w & M
     \\
  \end{array}
\right)
\end{equation*}
where $M$ is the period matrix of $P:=\prym\pa{C,\eta', R'}$ and, with the notations of \eqref{item:estMappa3} of Corollary \ref{cor:estmappa}, the vector $w$ represents essentially the class of $\pm[p_x-q_x]\in \mathcal{K}\pa{P}$. We consider now a filtration
\begin{equation*}
T''_{\mathcal{S}\pa{z}} \subset T'_{\mathcal{S}\pa{z}}\subset T_{\mathcal{S}\pa{z}},
\end{equation*}
analogous to that in \eqref{eq:primafiltr}, where $T_{\mathcal{S}\pa{z}}'$ is the tangent space of $\bar{\mathcal{A}}_{g-1+\frac{r}{2}}\setminus \mathcal{A}_{g-1+\frac{r}{2}}$ at $\mathcal{S}\pa{z}$, that is the space of the infinitesimal deformations of the matrix $\bar M$ such that the first entry is constantly zero. We define $T''_{\mathcal{S}\pa{z}}$ as the space of the infinitesimal deformations of $\bar M$ such that $M$ is constant, the first entry of $\bar M$ is zero and $w$ varies. We have again two exact sequences
\begin{align*}
0 \ra T'_{\mathcal{S}\pa{z}} &\ra T_{\mathcal{S}\pa{z}} \ra N_{\mathcal{S}\pa{z}} \ra 0,\\
0 \ra T''_{\mathcal{S}\pa{z}} &\ra T'_{\mathcal{S}\pa{z}} \ra N'_{\mathcal{S}\pa{z}} \ra 0,
\end{align*}
where $N_{\mathcal{S}\pa{z}}$ is the normal space of $\mathcal{A}_{g-1+\frac{r}{2}}$ in $\bar{\mathcal{A}}_{g-1+\frac{r}{2}}$ and $N'_{\mathcal{S}\pa{z}}$ is the space of the infinitesimal deformations of the compact Prym variety $P$.

The differential $d\mathcal{S}$ preserves the filtration and defines the following diagram
\[
\xymatrix{
0 \ar[r] & T''_{z}\ar[d]^{d\mathcal{S}|_{T''_{z}}} \ar[r] & T'_{z}\ar[d]^{d\mathcal{S}|_{T'_{z}}}  \ar[r] & N'_z\ar[r] \ar[d]^{d} &  0\\
0 \ar[r] & T''_{\mathcal{S}\pa{z}} \ar[r] & T'_{\mathcal{S}\pa{z}} \ar[r] & N'_{\mathcal{S}\pa{z}} \ar[r] & 0\\
}
\]
where $d$ is the map that makes the diagram commutative. We identify $d$ with the differential of the Prym map $\mathcal{R}_{g, r-2}\ra \mathcal{A}^\delta_{g-2+\frac{r}{2}}$ at a general point $\pa{C, \eta', R'}$. By Proposition \ref{prop:proplocaltorelli}, $d$ is injective. The restriction
\[
d\mathcal{S}|_{T''_z} \colon T''_z \ra T''_{\mathcal{S}\pa{z}}
\]
is the differential of the map
\begin{align*}
C &\ra \mathcal{K}\pa{P}\\
x &\mapsto \pm[p_x-q_x]
\end{align*}
at the point $x$. Since the map is finite on the image and $x$ is a general point of $C$, $d\mathcal{S}\pa{T''_z}\neq 0$. We can conclude that $d\mathcal{S}|_{T'_z}$ is injective.

To complete the proof, let $0\neq v\in T_z\setminus T'_z$; we claim that $d\mathcal{S}\pa{v}\neq 0$. Since $\dim T_z=\dim T'_z+1$, this implies  $\ker d\mathcal{S}\subset T'_z$. Consider the complex unit disk $\Delta$, a non constant map $\varphi \colon \Delta \ra \Upsilon$ such that $\varphi\pa{0}=z$ and $d\varphi\pa{\frac{d}{dt}}=v$, and the family of coverings \eqref{eq:famD_kcompl} described in Section \ref{subsec:Prymmapboundary}. We define the map
\[
\tau \colon \Delta\ra \bar{\mathcal{M}}_{2g-2+\frac{r}{2}}\\
\]
which sends $t$  to the class of isomorphism of $\mathcal{D}_t$.
Let
\[
T\colon \Delta  \ra \bar{\mathcal{A}}_{2g-1+\frac{r}{2}}\\
\]
be the composition of $\tau$ with the Torelli map. Note that $J\pa{\mathcal{D}_t}$ is isogenous to $J\pa{C}\times {\mathcal P}_t$ for each $t\in \Delta\setminus\{0\}$. Thus, in order to prove that $d\mathcal{S}\pa{v}\neq 0$, it is sufficient to show that $\frac{dT}{dt}\left.\right |_{t=0}\neq 0$. Since $d\tau_0\pa{v}$ is different from zero, this is a consequence of the infinitesimal Torelli theorem at the boundary (\cite{usui}).
\end{proof}

In order to compute the degree of $\mathcal{S}$ on $Y_2$, we need the following lemma.

\begin{lemma}
\label{lemma:nuovoPrym}
 Let $\pi\colon D \ra C$ be a double covering of a curve of genus $g\geq 1$ with $r>0$ branch points.
 \begin{enumerate}
  \item \label{item:abelprym} If $D$ is not hyperelliptic, then the Abel-Prym curve $D'\subset P$ determines the covering $\pi$.
  \item \label{item:Dnonhe} Assume that $\pi\colon D \ra C$ defines a general point in $\mathcal{R}_{g,r}$. The curve $D$ is hyperelliptic (and in particular the Abel-Prym map has degree one, see \eqref{item:AbelPrymPrelim} of Section \ref{subsec: preliminaries}) if and only if $g=1$ and $r=2$.
 \end{enumerate}
\end{lemma}
\begin{proof}
 \begin{enumerate}
  \item Let us consider the natural projection $\sigma\colon P \ra \mathcal{K}\pa{P}$ and set $C':=\sigma \pa{D'}$. The morphism $\sigma\colon D' \ra C'$ is a double covering. Passing to the normalization, we get $\pi\colon D \ra C$.
  \item If $g=1$ and $r=2$, then $g\pa{D}=2$ and $D$ is hyperelliptic. To prove the converse, assume by contradiction that $D$ is hyperelliptic and denote by $j\colon D \ra D$ the hyperelliptic involution and by $i\colon D \ra D$ the involution associated to $\pi$. Since $j$ commutes with $i$ (see e.g. \cite[Section III.8, Corollary 3]{FarkasKra}), it induces a non-trivial involution $j'\colon C \ra C$ that is invariant on the branch divisor of $\pi$. If $g>2$, we get a contradiction since a generic curve has only trivial automorphisms. Otherwise, when $g=2$ and $r\geq 2$ or $g=1$ and $r\geq 4$, it is always possible to find an effective divisor of degree $r$ that is not fixed by any involution.
 \end{enumerate}

\end{proof}

\begin{lemma}
\label{lemma:genin}
If $g>2$,
\[
\mathcal{S} \colon Y_2 \ra \mathcal{S}\pa{Y_2}\subset \bar{\mathcal{A}}_{g-1+\frac{r}{2}}^\delta
\]
is generically injective.
\end{lemma}
\begin{proof}
We recall that, by Corollary \ref{cor:estmappa}, a point
\[
\mathcal{S}\pa{\eta'\otimes\fascio{C}\pa{x}, R'+2x} \in \mathcal{S}\pa{Y_2}
\]
is determined by the Prym variety $P=\mathcal{P}\pa{C, \eta', R'}$ of dimension $\pa{g-1+\frac{r}{2}}-1$ and by the class $\pm[p_x-q_x]\in \mathcal{K}\pa{P}$. Let $P$ be a general Prym variety and let $V\subset \mathcal{S}\pa{Y_2}$ be the set of points that parametrize semi-abelian varieties of rank $1$ with compact part isomorphic to $P$. To prove the statement it is sufficient to prove that $\mathcal{S}$ is generically injective on $W:=\mathcal{S}^{-1}\pa{V}$.

Let us consider the Prym map
\[
\prym \colon \mathcal{R}_{g,r-2} \ra \mathcal{A}_{g-2+\frac{r}{2}}^\delta.
\]
By Corollary \ref{cor:Prymfinite}, there is only a finite number of points
\[
\set{\pa{C,\eta^k,R^k}\in \mathcal{R}_{g,r-2}}_{k=1, \ldots, n},
\]
such that $P =\prym \pa{C,\eta^k,R^k}$. Denote by $\pi^k \colon D^k \ra C$ the double covering of smooth curves associated to $\pa{C,\eta^k,R^k}$ and, for $x\in C$, set ${\pa{\pi^k}}^{-1}\pa{x}=\set{p^k_x,q^k_x}$.
We have
\[
W=\bigcup_{k=1}^nW^k,
\]
where
\[
W^k:=\set{\pa{\eta^k\otimes \fascio{C}\pa{x},R^k+2x}\in \Upsilon : x\in C \setminus \supp R^k}.
\]
To prove that $\mathcal{S}|_{W}$ is generically injective we show that, for each $i\neq j$, the curve $\mathcal{S}\pa{W^i}$ intersects $\mathcal{S}\pa{W^j}$ only in a finite number of points and that, for each $i$, the map $\mathcal{S}|_{W^i} \colon W^i \ra \mathcal{S}\pa{W^i}$ is generically injective.

We recall that
\begin{equation*}
\mathcal{S}\pa{\eta^i\otimes \fascio{C}\pa{x},R^i+2x}=\mathcal{S}\pa{\eta^j\otimes \fascio{C}\pa{y},R^j+2y}.
\end{equation*}
if and only if
\[
 \pm[p^i_x-q^i_x]=\pm[p^j_y-q^j_y]\in \mathcal{K}\pa{P}.
\]
It follows that $\ol {\mathcal{S}\pa{W^i}}=\ol {\mathcal{S}\pa{W^j}}$ if and only if the image of the Abel-Prym curve of $\pi^i$ in $\mathcal{K}\pa{P}$ coincides with that of $\pi^j$, that is (see Lemma \ref{lemma:nuovoPrym}) if and only if $i=j$. Furthermore, the map $\mathcal{S}|_{W^i}$ is generically injective (see \eqref{item:Dnonhe} of Lemma \ref{lemma:nuovoPrym}).
\end{proof}

\begin{lemma}
\label{lemma:genin2}
If $g=2$,
\[
\mathcal{S}\colon  Y_2 \ra \mathcal{S}\pa{Y_2}\subset \bar{\mathcal{A}}_{g-1+\frac{r}{2}}^\delta
\]
has generically degree two.
\end{lemma}
\begin{proof}
Let $i \colon C \ra C$ be the hyperelliptic involution. With the same notations used in Lemma \ref{lemma:genin}, we have
\[
W=\bigcup_{k=1}^nW^k \cup \bigcup_{k=1}^nW^k_*,
\]
where
\[
W_*^k:=\set{\pa{i^*\eta^k\otimes \fascio{C}\pa{x},i(R^k)+2x}\in \Upsilon : x\in C \setminus \supp i(R^k)}.
\]
Furthermore, for each $x$,
\[
\mathcal{S}\pa{\eta^k\otimes \fascio{C}\pa{x},R^k+2x}=\mathcal{S}\pa{i^*\eta^k\otimes \fascio{C}\pa{x},i(R^k)+2i\pa{x}}.
\]
The rest of the proof is analogous to that of Lemma \ref{lemma:genin}.
\end{proof}

\subsection{End of proof of Theorem \ref{theorem: globaltorelli}}

We are now ready to complete the proof of the Generic Torelli Theorem for Prym varieties of ramified coverings. We denote by 
\[
   \mathcal{U}_{C}:=\set{\pa{C, \eta, R}\in \mathcal{R}_{g,r}} \subset \mathcal{R}_{g,r}                                                                                                                              
\]
the moduli space of double coverings of a fixed curve $C$ with $r$ branch points.

\begin{lemma}
\label{lemma:torelliQspiegato}
Let $C$ be a general curve of genus $1<g<r$, with $r\geq 6$. Then the Prym map is generically injective on $\mathcal{U}_{C}$.
\end{lemma}
\begin{proof}
We recall that on the open set $Y_1\subset \Upsilon$ (see \eqref{eq:defD_k}) the map $\mathcal{S}$ coincides with the composition of $\mathcal{T} \colon \Upsilon \dashrightarrow \mathcal{R}_{g,r}$ (see \eqref{eq:razriv}) with the Prym map. Furthermore $\mathcal{T}$ is an isomorphism for $g>2$ and a map of degree two for $g=2$. Since $\mathcal{U}_{C}=\mathcal{T}\pa{Y_1}$, the statement follows from Proposition \ref{prop: torelliQ}.
\end{proof}

\begin{proof}[Proof of Theorem \ref{theorem: globaltorelli}]
If $r=4$ and $g=4$, the theorem is a direct consequence of \cite[Theorem 7.7]{pol1122}. By Corollary \ref{cor:gentorellifirst} the theorem is proved for $g\geq r$. Let us assume $g<r$ and consider a general point $y\in \prym\pa{\mathcal{R}_{g,r}}$. By Corollary \ref{cor:Prymfinite},
\[
\mathcal{P}^{-1}\pa{y}=\set{\pa{C_i,\eta_i,R_i}}_{i=1,\ldots,n}.
\]
Furthermore, since we can assume $y$ smooth in the Prym locus, by Corollary \ref{cor:gentorellifirst}, $C_i=C_j$ for each $i,j$. Finally, by Lemma \ref{lemma:torelliQspiegato}, $n=1$.
\end{proof}

\section{Applications}
\label{sec:applications}

\subsection{The Gauss map of a Prym variety}
\label{subsec:Gauss}

We recall the definition of the Gauss map: if $M$ is a complex torus of dimension $n+1$, then the tangent spaces $\set{T_xM}_{x\in M}$ are all naturally identified with $T_0M\simeq\mathbb{C}^{n+1}$. Let $X$ be an analytic subvariety of dimension $k+1$ of $X$ and denote by $X_{\mathrm{ns}}$ the smooth locus of $X$ and by $\mathbb{G}\pa{k,n}$ the Grassmanian of the $k$-planes in $\mathbb{P}^n$. The \emph{Gauss map} of $X$ in $M$ is the map
\begin{align*}
  X_{\mathrm{ns}} &\ra \mathbb{G}\pa{k,n}\\
x &\mapsto \mathbb{P}T_x\pa{X}\subset \mathbb{P}T_x\pa{M}=\mathbb{P}^n.
\end{align*}

Note that a Prym variety $P$ is not principally polarized; in fact if $L$ is the polarization on $P$, then $h^0\pa{P, \fascio{P}\pa{\Theta_P}}=2^g$, where $g$ is the genus of $C$. Thus, given $P$, there is no natural choice of divisor in the associated linear system. However, we have proved in Theorem \ref{theorem: globaltorelli} that a general Prym variety $P$ arises from a unique covering $\pi\colon D \ra C$, and so in this case we may assign to $P$ the divisor $\Theta_P:=\Theta_{J\pa{D}}\cap P$, which is well defined up to translation. In analogy with the case of Jacobian varieties of hyperelliptic curves (see \cite{andreotti}), the pair $\pa{P, \Theta_P}$ allows us to determine the branch locus of the covering $\pi$. In fact we have:
\begin{prop}
\label{prop: Gaussmap}
Let $P:=\prym\pa{C, \eta, R}$ be the Prym variety of $\pi \colon D\ra C$. If we identify the tangent space at any point of $P$ with $H^0\pa{C, \omega_C \otimes \eta}^*$, then the branch locus $\mathfrak{B}$ of the Gauss map
\[
\pa{\Theta_P}_{\mathrm{ns}}\ra \mathbb{P}H^0\pa{C, \omega_C \otimes \eta},
\]
of $\Theta_P$ in $P$, is dual to the branch locus of $f_\eta \circ \pi$, where $f_\eta \colon C_\eta \ra \mathbb{P}H^0\pa{C, \omega_C \otimes \eta}^*$ is the semi-canonical map. In particular, if $\omega_C\otimes\eta$ is very ample, $\mathfrak{B}$ consists of $r$ distinct hyperplanes.
\begin{proof}
Set $d:=2g-2+\frac{r}{2}=g\pa{D}-1$ and define
\[
\hat{P}:=\pi^{-1}\pa{\omega_C\otimes \eta}\subset \pic^d\pa{D},\qquad \Theta_{\hat{P}}:= \tilde\Theta\cap \hat{P},
\]
where $\tilde\Theta$ is the set of the effective line bundles of degree $d$ on $D$. Due to the isomorphism between $J\pa{C}$ and $\pic^d\pa{C}$ that maps $0$ to $\omega_C\otimes \eta$, we can identify $(\hat P, \Theta_{\hat{P}})$ with $\pa{P, \Theta_{P}}$. Consider the vector spaces
\[
H:=H^0\pa{D,\omega_{D}}^*, \qquad H^+, \qquad H^{-},
\]
where $H^+$ and $H^{-}$ are the eigenspaces of $1$ and $-1$ for the involution induced on $H$ by $\pi$. The map $\pi^*$ identifies canonically $H^+$ and $H^{-}$ with $H^0\pa{C, \omega_C}^*$ and $H^0\pa{C,\omega_C\otimes \eta}^*$ respectively (cf. \cite{verra} for the \'etale case). Thus we have the following commutative diagram
\[
\xymatrix{
D \ar[d]_{\pi} \ar[rr]^{f_{\omega_D}} &&\mathbb{P}H\ar@{-->}[d]^{h^{-}} \\
C \ar[rr]^{f_\eta} && \mathbb{P}{H^-}
}
\]
where $f_{\omega_D}$ is the canonical map of $D$ and $h^{-}$ is the projection of centre $\mathbb{P}{H^+}$. Notice that $H$ can be identified with the tangent space of $\pic^d\pa{D}$ at any point and, consequently, $H^{-}$ is the tangent space to $\hat P$.

Denote by $\mathcal{G}$ the Gauss map of $\Theta_{\hat{P}}$ in $\hat P$ and  by $\mathcal{G}'$ the Gauss map of $\tilde \Theta$ in $\pic^d\pa{D}$. Given $\sum_{j=1}^{d} p_j\in \pa{\Theta_{\hat{P}}}_{\mathrm{ns}}$, we have $\sum_{j=1}^{d} p_j\in \tilde \Theta_{\mathrm{ns}}$, and $\mathcal{G}\pa{\sum_{j=1}^{d} p_j}$ is simply the projection of $\mathcal{G}'\pa{\sum_{j=1}^{d} p_j}$ (the hyperplane spanned by the points $\set{f_{\omega_D}\pa{p_j}}_{j=1}^d$) under $h^-$. We can conclude that $\mathcal{G}\pa{\sum_{j=1}^{d} p_j}$ is the hyperplane of $\mathbb{P}{H^-}$ which intersects the semi-canonical model $C_\eta$ in the points $\set{f_\eta\circ\pi\pa{p_j}=h^{-}\circ f_{\omega_D}\pa{p_j}}_{j=1}^d$.
\end{proof}
\end{prop}

\begin{rmk}
Observe that, if $r\geq 6$, then $\omega_C\otimes \eta$ is very ample. When this condition is not satisfied, the description of the branch locus $\mathfrak{B}$ is trickier. If, for example, $r=4$ and $g\geq 2$, there are two more possible situations:
\begin{enumerate}
  \item If $h^0\pa{\eta}=1$ and $\eta=p+q$ (resp. $\eta=2p$), the semi-canonical model $C_\eta$ is the nodal curve obtained from $C$ by identifying $p$ and $q$ (resp. a cuspidal curve). The branch locus $\mathfrak{B}$ is made of $4$ (resp. $5$) hyperplanes counted with multiplicities.
   \item If $h^0\pa{\eta}=2$, $C$ is hyperelliptic and the linear system $\abs{\eta}$ defines the $g^1_2$. Then $C_\eta$ is a rational normal curve and $\deg f_\eta =2$. The branch locus $\mathfrak{B}$ is made of $4g+8$ hyperplanes counted with multiplicities: there are $2g+2$ hyperplanes of multiplicity $2$, that are dual to the branch points of $f_\eta$, and a hyperplane with multiplicity $4$ (possibly not distinct from the previous ones) that is dual to the image of the branch points of $\pi$ in $C_\eta$. 
\end{enumerate}
Note that also in these cases, it is always possible to recover the branch locus of $\pi$ from $\mathfrak{B}$.
\end{rmk}

\subsection{Proof of Theorem \ref{theorem:isogjacob}}
\label{subsec:prymjac}

In this section we prove that a very general Prym variety of dimension at least $4$ is not isogenous to a Jacobian variety. For some values of $g$ and $r$ the statement is a simple consequence of the fact that the Prym locus has dimension larger than the moduli space of Jacobian varieties (see \eqref{item:PrymLocus1} of Proposition \ref{prop:Prymlocus}). In the other cases the moduli count shows that a general Prym variety is not isogenous to a Jacobian of a hyperelliptic curve (see \eqref{item:PrymLocus2} of Proposition \ref{prop:Prymlocus}). Using this fact and degeneration techniques, like in the proof of Theorem \ref{theorem: globaltorelli}, we can conclude the proof of Theorem \ref{theorem:isogjacob}.

\begin{prop}
\label{prop:Prymlocus}
Let $P$ be a general Prym variety in the Prym locus $\mathcal{P}^{\delta}_{g-1+\frac{r}{2}}$.
\begin{enumerate}
\item \label{item:PrymLocus1} If
\[
r=4 \text{ and } g\geq 3, \qquad \text{ or} \qquad r=2 \text{ and } g\geq 4,
\]
$P$ is not isogenous to a Jacobian variety.
\item \label{item:PrymLocus2} If
\[
r\geq 6 \text{ and } g\geq 1, \qquad \text{ or} \qquad r=4 \text{ and } g=2,
\]
$P$ is not isogenous to a Jacobian of a hyperelliptic curve.
\end{enumerate}
\end{prop}
\begin{rmk}
Notice that, if $r=6$ and $g=1$ or $r=4$ and $g=2$, $P$ has dimension $3$ and, consequently, it is isogenous to a Jacobian.
\end{rmk}
\begin{proof}
We compare the dimension of the Prym locus $\mathcal{P}^{\delta}_{g-1+\frac{r}{2}}$ and the Jacobian locus (resp. hyperelliptic locus). When the Prym map is generically finite (see Corollary \ref{cor:Prymfinite}), the result follows by a count of parameters. Hence we have only to consider the cases $r=2$ and $g=4$, or $r=4$ and $g=2$. We claim that in these cases the differential of the Prym map is generically surjective, this implies that
\[
\dim \mathcal{P}^{\delta}_{g-1+\frac{r}{2}}=\dim \mathcal{A}_{g-1+\frac{r}{2}}^\delta.
\]
To see this we show that the co-differential is injective. We recall (see Section \ref{subsec:quadrics}) that the co-differential of the Prym map
\begin{equation*}
d\prym ^* \colon \sym^2H^0\pa{C, \omega_C\otimes \eta} \ra H^0\pa{C, \omega_C^2\otimes\fascio{C}\pa{R}}.
\end{equation*}
is injective if and only if the semi-canonical model $C_\eta$ of $C$ is not contained in any quadric. If $r=4$ and $g=2$ the statement follows from the fact that $C_\eta$ is a plane curve of degree $4$. In the other case, it is a consequence of the following lemma.
\end{proof}

\begin{lemma}
\label{lemma:grado7quadriche}
Let $C$ be a non-hyperelliptic curve of genus $4$ and $\eta\in \pic^1\pa{C}$ be a general line bundle of degree $1$ on $C$ such that $h^0\pa{\eta^2}>0$. Then the image $C_\eta$ of the semi-canonical map
\[
f_\eta \colon C \ra \mathbb{P}^{3}
\]
does not lie on any quadric.
\end{lemma}
\begin{proof}
We recall that, since $C$ is not hyperelliptic, $\omega_C\simeq L_1\otimes L_2$, where $\deg L_i=3$, $h^0\pa{L_i}=2$ and the line bundles can possibly coincide (\cite[Chapter 5]{acgh}). As in Lemma \ref{lemma:tecnico}, one can prove that the semi-canonical map is an embedding. Assume by contradiction that $C_\eta$ lies on a quadric $Q$. If $Q$ is isomorphic to $\mathbb{P}^1\times\mathbb{P}^1$, by the adjunction formula, since $\deg C_\eta =7$, we have that the bidegree $\pa{d_1,d_2}$ of $C_\eta$ satisfies the following relations
\begin{align*}
\pa{d_1-1}\pa{d_2-1}&=4,\\
d_1+d_2=7.
\end{align*}
Thus, either $d_1=2$ or $d_2=2$, that is $C$ is hyperelliptic, in contradiction with the hypothesis. If $Q$ is a cone, let $g \colon C \ra \mathbb{P}^2$ be the composition of $f_\eta$ with the projection from the vertex of the cone. The image of $g$ is a conic and so $g$ factors as
\[
C \ra \mathbb{P}^1 \xrightarrow{\abs{\fascio{\mathbb{P}^1}\pa{1}}} \mathbb{P}^2,
\]
where the first morphism has degree $2\leq d \leq 3$ and the second one has degree $2$. Since $C$ is not hyperelliptic, $\deg g =6$. It follows that $g$ is the map associated to $\abs{\omega_C\otimes\eta\otimes \fascio{C}\pa{-p}}$, or, equivalently, that $p\in C_\eta$ is the vertex of the cone. By the previous discussion we can conclude that
\[
\omega_C\otimes\eta\otimes\fascio{C}\pa{-p}\simeq L_1^2
\]
and so, since $\omega_C\simeq L_1\otimes L_2$,
\[
\eta\simeq L_1\otimes L_2^{-1}\otimes\fascio{C}\pa{p}.
\]
If $L_1\simeq L_2$, this implies that $\eta$ is effective and we get a contradiction. Otherwise, $\eta$ varies at most in a one dimensional family. On the other hand, by hypothesis, $\eta$ depends on two parameters. This yields a contradiction.
\end{proof}

\begin{rmk}
We notice that the argument of Proposition \ref{prop:Prymlocus} shows that, also in the \'etale case, a general Prym variety of dimension greater or equal than 4 is not isogenous to a Jacobian variety.
\end{rmk}

To complete the proof of Theorem \ref{theorem:isogjacob} we need the following lemma concerning the difference surface (see \eqref{item:diff} in Section \ref{subsec: preliminaries}):

\begin{lemma}
\label{lemma:diffpern}
Let $C$ be a non-hyperelliptic curve and $n\in \mathbb{N}$ be a non zero integer. 
\begin{enumerate}
\item \label{item:diff1} $n\Gamma_C$ is birational to $C\times C$.
\item \label{item:diff2} $n\Gamma'_C$ is birational to the symmetric product $C_2$ of $C$.
\end{enumerate}
In particular, $\Gamma_C$ is birational to $C\times C$ and $\Gamma'_C$ is birational to $C_2$.
\end{lemma}
\begin{proof}
Arguing as in \cite[Lemma 3.1.1, Proposition 3.2.1]{artPirola}, we can assume $n=1$. To prove \eqref{item:diff1}, notice that, if for the general point $\pa{a,b}\in C\times C$ there exists $\pa{c,d}\in C\times C$ such that $[a-b]=[c-d]$, then $C$ is hyperelliptic. Statement  \eqref{item:diff2} follows from \eqref{item:diff1}.
\end{proof}

\begin{proof}[Proof of Theorem \ref{theorem:isogjacob}]
By Proposition \ref{prop:Prymlocus}, we have to consider only the cases $r\geq 8$ and $g\geq 1$ or $r=6$ and $g\geq 2$. We assume by contradiction that a very general Prym variety is isogenous to a Jacobian. Then there exists a map of families
\begin{equation}
\label{eq:isogsututto}
\mathcal{J}/\mathcal{U} \xrightarrow{G} \mathcal{P}/\mathcal{U},
\end{equation}
where $\mathcal{U}$ is a finite \'etale covering of a dense open subset of $\mathcal{R}_{g,r}$ and, for each $t\in \mathcal{U}$,
\begin{enumerate}
\item $\mathcal{P}_t$ is a Prym variety of dimension $g-1+\frac{r}{2}$;
\item $\mathcal{J}_t$ is the Jacobian of a curve of genus $g-1+\frac{r}{2}$;
\item $G_t \colon \mathcal{J}_t \ra \mathcal{P}_t$ is a surjective morphism of abelian varieties.
\end{enumerate}

\emph{Step I: limits of $G$ (\eqref{eq:isogsututto}) at the boundary.}

\noindent We want to extend the map $G$ to some point of the boundary of $\Upsilon$ (see \eqref{eq:superf}) associated to an admissible covering as in Figure \ref{disegno} of page \pageref{disegno}.

Let $\pi \colon E \ra C$ be a very general double covering of a curve of genus $g$ with $r-2$ branch points, and assume that $\pi$ is determined by the triple $\pa{C, \eta', R'}$. For each non-branch point $x\in C$, we consider a family of admissible coverings
\begin{equation*}
\label{diag: curve disco}
\mathcal{D}^x/\Delta  \rightarrow  \mathcal{C}^x/\Delta ,
\end{equation*}
obtained as in Section \ref{subsec:Prymmapboundary} (see \eqref{eq:famD_kcompl}) by choosing a unit disk centred in the point $\pa{\eta'\otimes\fascio{C}\pa{x}, R'+2x}$.
Let us restrict our initial map of families \eqref{eq:isogsututto} to $\Delta\setminus\set {0}$. Changing base, if necessary, by completion, we obtain a map of families
\[
\mathcal{J}^x/\Delta \ra \mathcal{P}^x/\Delta.
\]
The semi-abelian variety $\mathcal{P}_0^x$ is the kernel of the morphism $J\pa{\mathcal{D}_0^x}\ra J\pa{\mathcal{C}_0^x}$, where $J\pa{\mathcal{C}_0^x}= J\pa{C}$, $J\pa{\mathcal{D}_0^x} = J\pa{E_x}$, and $E_x$ is the singular curve obtained from $E$ by identifying $p_x$ and $q_x$, where $\pi^{-1}\pa{x}=\set{p_x,q_x}$. We denote by $P$ the compact part of $\mathcal{P}_0^x$. The semi-abelian variety $\mathcal{J}_0^x$ is a generalized Jacobian variety of a singular curve $H_x$, obtained from $H$ by identifying two distinct points $a_x$ and $b_x$. We denote by $\varphi\colon J\pa{H} \rightarrow P$ the isogeny induced on the compact quotients.

\emph{Step II: comparing the extension classes of $E$ and $H$.}

\noindent  By varying $x\in C$, we can perform different degenerations of the families in \eqref{eq:isogsututto}. Notice that the compact quotient $P$ of $\mathcal{P}_0^x$ does not depend on $x$. It follows that also the normalization $H$ of  $H_x$ and the isogeny $\varphi\colon J\pa{H} \ra P$ are independent of the chosen degeneration. Thus for each non-branch point $x\in C$, there exists $a_x,b_x\in H$ such that
\begin{equation}
\label{eq:compat}
\varphi^*\pa{[p_x-q_x]}=n_x[a_x-b_x],
\end{equation}
for some $n_x$ different from zero.

\emph{Step III: conclusion.}
\noindent Let us consider the diagram
\[
\xymatrix{
 P\ar[d]^{\sigma}\ar[r]^{\psi \qquad} & \pic^0\pa{P}\ar[d]^{\sigma^\circ}\ar[r]^{\varphi^*\quad} & \pic^0\pa{J\pa{H}}\ar[d]^{\sigma_H^{\circ}} \\
\mathcal{K}\pa{P}\ar[r]^{\psi_{\mathcal{K}}}& \mathcal{K}^0\pa{P}\ar[r]^{\varphi^*_\mathcal{K}} & \mathcal{K}^0\pa{J\pa{H}} \\
}
\]
where $\psi$ is the isogeny induced by the polarization, the vertical arrows are the natural projections, and $\psi_{\mathcal{K}}$ and $\varphi^*_\mathcal{K}$ are the maps induced respectively by $\psi$ and $\varphi^*$ on the Kummer varieties. Set
\begin{align*}
E''&= \psi\pa{E'}\subset \pic^0\pa{P} ,\\
C''&= \psi_{\mathcal{K}}\pa{C'}=\sigma^\circ\pa{E''}\subset  \mathcal{K}^0\pa{P},
\end{align*}
 where $E'$ is the Abel-Prym curve of $\pi$ and $C'$ is its projection in the Kummer variety. Denote by $\Gamma_H$ the image of the difference map in $\pic^0\pa{J\pa{H}}$ (see \eqref{item:diff} in Section \ref{subsec: preliminaries}) and by $\Gamma'_H$ its projection in the Kummer variety. By \eqref{eq:compat}, arguing as in \cite[Section 2]{artPirola}, we find $n\neq 0$ such that $\varphi^*\pa{E''}\subseteq n \Gamma_H$. It follows that \[
  \varphi^*_{\mathcal{K}}\pa{C''}\subseteq n \Gamma'_H.                                                                                                                                                                                                                                                                                                                                                                                                                                                                                                                                                                                                                                                                                                                                                                                                                                                                                                         \]

Since $H$ is not hyperelliptic (Proposition \ref{prop:Prymlocus}), by Proposition \ref{lemma:diffpern}, $n \Gamma'_H$ is birational to the symmetric product $H_{2}$ of $H$. By composition, we obtain a non-constant rational map
\[
C'\xrightarrow{\varphi_{\mathcal{K}}} C'' \xrightarrow{\varphi^*_{\mathcal{K}}} \varphi^*_{\mathcal{K}}\pa{C''} \hookrightarrow n \Gamma'_H  \dashrightarrow H_2 \hookrightarrow J\pa{H}.
\]
We notice that, since by hypothesis $r\geq 6$, it holds $g<g-2+\frac{r}{2}=\dim P= \dim J\pa{H}$. Moreover, the geometric genus of $C'$ is at most $g$. Thus we can conclude that $J\pa{H}$ is not simple. On the other hand, $P$ is very general and so, by \cite{HodgePrym}, the N\'eron-Severi group $\ns\pa{P}$ is isomorphic to $\mathbb{Z}$. It follows that $P$, and consequently $J\pa{H}$, is simple. Thus we get a contradiction.
\end{proof}

\begin{acknowledgements}
We would like to thank Juan Carlos Naranjo for many helpful discussions during his visit to Pavia in June 2010, and Enrico Schlesinger for his helpful suggestion on complete intersections. We also wish to thank the referee whose comments helped improve the exposition.
\end{acknowledgements}

\end{document}